\documentclass{amsart}
\usepackage[T1]{fontenc}
\usepackage[utf8]{inputenc}
\usepackage{amsmath,amsthm,amsfonts,amscd,amssymb,eucal,latexsym,mathrsfs}
\usepackage{stmaryrd}
\usepackage{enumerate}
\usepackage{hyperref}
\usepackage[all]{xy}
\usepackage{etoolbox}

\usepackage{tikz,tikz-cd}
\usetikzlibrary{shapes.geometric}
\usetikzlibrary{arrows}

\usetikzlibrary{positioning}
\usepackage{float}
\usepackage{MnSymbol}
\usetikzlibrary{matrix}

\newcommand{\toc}{\tableofcontents}

\theoremstyle{plain}
\newtheorem{theorem}{Theorem}[section]
\newtheorem*{theorem*}{Theorem}

\newtheorem{corollary}[theorem]{Corollary}
\newtheorem*{corollary*}{Corollary}

\newtheorem{lemma}[theorem]{Lemma}

\newtheorem{proposition}[theorem]{Proposition}

\theoremstyle{definition}
\newtheorem{remark}[theorem]{Remark}

\newtheorem{example}[theorem]{Example}

\newtheorem{definition}[theorem]{Definition}
\newtheorem*{definition*}{Definition}

\usetikzlibrary{calc,graphs}
\usepackage{xcolor}
\usetikzlibrary{arrows,decorations.pathmorphing}

\newcommand{\p}{\varphi}

\newcommand{\e}{\varepsilon}

\newcommand{\IR}{\mathbb{R}}

\DeclareMathOperator{\del}{\partial}

\DeclareMathOperator{\inj}{\hookrightarrow}

\newcommand{\SU}{\mathrm{SU}}

\newcommand{\ts}{\textsection}

\newcommand{\rk}{\mathrm{rk}}
\newcommand{\val}{\mathrm{val}}
\newcommand{\Id}{\mathrm{Id}}
\newcommand{\Aut}{\mathrm{Aut}}
\newcommand{\up}{\mathrm{\uparrow}}

\title{Root distributions in Moebius--Kantor complexes}

\author{Sylvain Barr\'e}
\author{Mika\"el Pichot}
\address{Sylvain Barr\'e, UMR 6205, LMBA, Université de Bretagne-Sud,BP 573, 56017, Vannes, France}\email{Sylvain.Barre@univ-ubs.fr}
\address{Mika\"el Pichot, McGill University, 805 Sherbrooke St W., Montr\'eal, QC H3A 0B9, Canada}\email{pichot@math.mcgill.ca}

\begin{document}

\begin{abstract}
We study the distribution of roots of rank 2 in nonpositively curved 2-complexes with Moebius–Kantor links. For every face in such a complex, the parity of the number of roots of rank 2 in a neighbourhood of the face is a well-defined geometric invariant determined by the root distribution. We study the relation between the root distribution and the parity distribution.
We prove that there exist parity distributions in flats which are disallowed in Moebius--Kantor complexes. This contrasts with the fact that every root distribution can be realized. We classify the root distributions associated with an even parity distribution (i.e., such that every face is even) on a flat plane. We prove that there exists up to isomorphism a unique even simply connected Moebius--Kantor complex---namely, the Pauli complex.
\end{abstract}

\maketitle

\section{Introduction}

A Moebius--Kantor complex is a 2-complex with triangle faces whose links are  isomorphic to the Moebius--Kantor graph (i.e., the unique  cubic symmetric graph with 16 vertices). Every Moebius--Kantor complex can be viewed as  a nonpositively curved  2-complex, in which every face is isometric to an equilateral triangle with sides of length 1 \cite[II.5.1]{BH}.  

We shall begin by reviewing a few  definitions. 

\subsection{Roots}
Let $\Delta$ be a nonpositively curved 2-complex, $x\in \Delta$ be a vertex, and $L_x$ the link at $x$, endowed with the angular metric. We call \emph{root} at $x$ a isometric embedding $\alpha\colon [0,\pi]\inj L_x$ such that $\alpha(0)$ is a link vertex (of order $>2$). Every root has a \emph{rank} $\rk(\alpha)$, which is a rational number in $[1,2]$ defined by
\[
\rk(\alpha):=1+{N(\alpha)\over q_\alpha}
\] 
where 
\[
N(\alpha):=|\{\beta\in \Phi_x\mid \alpha\neq \beta, \alpha(0)=\beta(0), \alpha(\pi)=\beta(\pi)\}|,
\]
$\Phi_x$ denotes the set of roots at $x$ and, for a root $\alpha$,  $q_\alpha$ denotes the order of $\alpha(0)$ minus 1:
\[
q_\alpha:=\val(\alpha(0))-1.
\]  
We refer to \cite[\ts 4]{chambers} for more details on these definitions.

\subsection{Parity in Moebius--Kantor complexes}
Let $\Delta$ be a Moebius--Kantor complex. Let $f$ be a face in $\Delta$. The parity of $f$ is defined as follows (cf.\  \cite[\ts 2]{parity}). Let $\tilde f$ denote the equilateral triangle formed by the union of $f$ and three faces in $\Delta$ corresponding to a choice of a face  $\neq f$ adjacent to every side of $f$. We call $\tilde f$ a \emph{large triangle} containing $f$.
 For every vertex $x$ of $f$, the triangle $\tilde f$ determine a root $\alpha_x$ at $x$ in $\Delta$. 

\begin{definition} We call \emph{parity} of $f$ the parity of the number of roots $\alpha_x$ which are of rank 2, when $x$ runs over the three vertices of the face $f$.
\end{definition}

This is a well-defined invariant, i.e., the parity of $f$ does not depend on the choice of the large triangle $\tilde f$ containing $f$, by \cite[Lemma 2.1]{parity}.

\subsection{Root distributions} \label{S - Intro root distribution}
Let $\Pi$ be a flat plane in a Moebius--Kantor complex $\Delta$, i.e., the image of an isometric embedding $ \IR^2\to \Delta$, where $\IR^2$ is endowed with the Euclidean metric.  Every flat is tessellated by equilateral triangles with sides of length 1. 

Let $x\in \Pi$ be a vertex. For every simplicial line $\ell\subset \Pi$ through $x$, the rank of every root $\alpha\in \Phi_x$  whose image is included in $\Pi$ and whose extremities belong to $\ell$ is constant.  We call $\ell$ a simplicial direction at $x$, and the common value $\rk(\alpha)$ the rank of $\ell$. Observe that precisely two of the three simplicial directions at $x$ in $\Pi$ are of rank 2.  

\begin{definition}
The \emph{root distribution of $\Pi$} is the map $\delta$ which associates to every vertex $x\in\Pi$ the unique simplicial direction at $x$ which is not of rank 2.  
\end{definition}

Graphically, one may represent $\delta$ as a field of segments of length $2\e$ which takes a vertex $x$ to the segment $(x-\e,x+\e)$ containing $x$ in $\delta(x)$, for some small fixed $\e>0$. 

We call \emph{abstract root distribution} on $\IR^2$  the choice, for every vertex $x$ of the tessellation by equilateral triangle, of a simplicial direction $\delta(x)$ at $x$.
 \begin{definition} 
An abstract root distribution is \emph{realized} in a Moebius--Kantor complex $\Delta$ if it is the root distribution of some flat $\Pi$ in $\Delta$. 
\end{definition}

We shall first prove that every abstract root distribution can be realized in some Moebius--Kantor complex:

\begin{proposition}\label{L - abstract root distribution in flat}
Every abstract root distribution can be realized in a Moebius--Kantor complex.  
\end{proposition}

We refer to \ts \ref{S - root prescription in flats} for the proof. This result does not hold for abstract parity distributions (see Cor.\   \ref{T - abstract parity not realizable}).

\subsection{Parity prescriptions in a flat} Every abstract root distribution $\delta$ on $\IR^2$ induces a parity distribution $p$ on the triangle faces of the tessellation. If $f$ is a face, then the parity $p(f)\in \{0,1\}=\{\text{even, odd}\}$ is the parity of the image of $f$ under an isometric embedding $ \IR^2\to \Delta$ which realizes $\delta$ in some Moebius--Kantor complex $\Delta$ (which exists by Prop.\ \ref{L - abstract root distribution in flat}). Since the value of $p(f)$ does not depend on the choice of the large triangle $\tilde f$, one may  always assume that $\tilde f$ is included in the image of $ \IR^2\to \Delta$. Thus, $p(f)$ depends only on the abstract root distribution $\delta$. It computes the parity of the number of sides of $\tilde f$ which are not in the simplicial directions selected by $\delta$.    

\begin{definition}
We call $p$ the \emph{parity distribution of $\delta$}. 
\end{definition}

An \emph{abstract parity distribution} on $\IR^2$  is the choice, for every face $f$ of the tessellation by equilateral triangle, of  value a $p(f)\in \{0,1\}$.

We aim to show (in   contrast with Prop.\ \ref{L - abstract root distribution in flat}) that there exist abstract parity distributions which are not the parity distribution of any flat in any Moebius--Kantor complex.   This will be an immediate corollary of the following result.

\begin{theorem}\label{L - abstract parity distribution in flat}
There exists an abstract parity distribution which is not the parity distribution of an abstract root distribution.
\end{theorem}

We refer to \ts\ref{S - parity root} for the proof.

An abstract parity distribution is said to be \emph{realized} in a Moebius--Kantor complex $\Delta$ if it is the induced parity distribution of some flat in $\Delta$.

\begin{corollary}\label{T - abstract parity not realizable}
There exists an abstract parity distribution on $\IR^2$ which is not realized in a Moebius--Kantor complex.
\end{corollary}

\begin{proof} The parity distribution given by Theorem \ref{L - abstract parity distribution in flat} is not  realized in a Moebius-Kantor complex $\Delta$, for it would be the parity distribution of the root distribution induced by $\Delta$.
\end{proof}

On the other hand, we can prove the following result. 

\begin{proposition}\label{P - balls of radius 1}
Every parity distribution on a ball of radius 1 (an hexagon) in $\IR^2$ can be realized by an abstract root distribution, and therefore can be realized in a Moebius--Kantor complex. 
\end{proposition}

This is established in \ts\ref{S - parity one ball}.

\subsection{Parity prescriptions in a Moebius-Kantor complex} By Theorem \ref{L - abstract parity distribution in flat}, there exists an abstract parity distribution on $\IR^2$ which cannot realized in a Moebius--Kantor complex. This should be compared with the results in \ts 6 of \cite{parity}, according to which the  parity can be  freely prescribed in Moebius--Kantor complexes, and therefore  that any abstract parity distribution can always ``be realized''  in some Moebius--Kantor complex.  

The point here is that in the case of a parity prescription in a Moebius--Kantor complex, the parity distribution can only be pre-assigned and then realized inductively in successive spheres, as opposed to being prescribed in advance as a 0-1 valued function (since there does not exist a predefined space that would serve as domain for this function). 

The free prescription theorem (from \cite[\ts 5]{parity}) for Moebius--Kantor complexes can be stated as follows.  Let $B_n$ denote a ball of radius $n$ centered at a vertex in a Moebius--Kantor complex, and  $S_n:=B_n\setminus  (B_{n-1})^\circ$ denote the closed simplicial sphere of radius $n$, where $(B_{n-1})^\circ$ denotes the interior of $B_{n-1}$. One can realize an arbitrary parity function on $S^n$ in a Moebius--Kantor complex by iterating the following construction. 

\begin{theorem}\label{T - free prescription} Given an arbitrary function $p\colon S_n\to\{0,1\}$ defined on the 2-skeleton of $S_n$,
there exists a ball $B_{n+1}$ in a Moebius--Kantor complex, containing $B_n$,  such that the parity of the faces of $S_n$ in $B_{n+1}$ is given by $p$.
\end{theorem}

Free construction theorems of this sort originated in the theory of buildings, in particular in works of Ronan and Tits (see for instance \cite{ronan1986construction,ronan1987building,tits1974buildings,tits1981local} and references therein).  

It was mentioned in \cite{parity} that the proof of Theorem \ref{T - free prescription} seems to provide additional free parameters in the constructions, i.e., that the parity function $p$ should not determine the isomorphism type of $B_{n+1}$ in general.  That these additional degrees of freedom ultimately exist depends again on the balance between the root distribution at the  vertices in the boundary of $B_n$, and the parity distribution on the faces of $S_n$.  
A direct computation can be made which shows, very roughly, what the degree of freedom  for the parity is with respect to the degree of freedom for the roots. 

Namely, let $f_n$ denote the number of faces in $S_n$. There can be two types of vertices in the boundary of $B_n$, those of order 3 and those of order 4. Let $a_n$ and $b_n$ denote their respective number. Then the relative degree of freedom is computed as follows: there are $2^{f_n}$ choices for the parity in $S_n$, and at most $8^{a_n}16^{b_n}$ choices for the roots at the boundary of $B_n$;  the quantities $a_n,b_n,f_n$ are related by the equality:
\[
f_{n+1}=18a_n+15b_n;
\]
the quantities $a_n$ and $b_n$ can be computed explicitly by the recurrence relation:
\[
\begin{pmatrix} 
a_{n+1} \\
b_{n+1} 
\end{pmatrix}=
\begin{pmatrix} 
6 & 2 \\
3 & 4 
\end{pmatrix}
\begin{pmatrix} 
a_n \\
b_n 
\end{pmatrix}.
\]
Therefore,
\[
2^{f_n}<8^{a_n}16^{b_n}.
\]  
Since we do not need them except for a general estimate of the degrees of freedom, we leave to the reader the exercise of verifying these relations.

A Moebius--Kantor complex is said to be \emph{even} if every face is even. In \cite[Remark 5.3(1)]{parity}, it was suggested to study the case of even complexes, and that---in view of the above estimates---there ``ought to exist'' uncountably many non isomorphic constructions of such complexes. It turns out that there exists  a unique construction up to isomorphism. In the remainder of this introduction, we shall state this uniqueness result more precisely, together with some applications which resolve some questions that we left open in earlier investigations. 

\subsection{Even root distributions}\label{S - even root distr} We shall first describe the relation between root distributions and parity distributions in the even case.

The even parity distribution on $\IR^2$ is the contant distribution $p\equiv 0$ in which every face is even. A root distribution is said to be even if its parity distribution is the even parity distribution.

We say that two abstract root distribution $\delta$ and $\delta'$ on $\IR^2$ are isomorphic if there exists a simplicial isometry $\alpha$ of $\IR^2$ such that $\delta\circ \alpha=\delta'$.

\begin{theorem}\label{T - even root distribution classified}
The even root distributions  can be classified up to isomorphism. 
\end{theorem}

 This result can be proved directly, or it can be deduced from \cite{Pauli}. We provide both arguments in the present paper. The second approach is shorter since it only requires to provide a few additions  to the results of \cite{Pauli}. The direct approach is given in \ts\ref{S - classifying even roots}. 

In \cite{Pauli},  a Moebius--Kantor complex $\Delta_P$ called the \emph{Pauli complex} was introduced. 
It is defined using a triangle of groups construction associated with the group generated by the Pauli matrices in $\SU(2)$. The fundamental group of this triangle of group is a group $G_P$ which is developable and therefore acts  properly isometrically on a CAT(0) complex, the complex $\Delta_P$. (More details on this construction are given in \ts \ref{Pauli} below.)

\begin{proposition}
The Pauli complex $\Delta_P$ is even. 
\end{proposition}

This proposition is established in \ts\ref{Pauli} below; it relies on basic  relations verified by the Pauli matrices. 

In order to deduce Theorem \ref{T - even root distribution classified} from \cite{Pauli}, it is sufficient to prove the following result (see \ts \ref{Pauli}).

\begin{proposition}\label{P - puzzle surj}
Every even root distribution can be realized in the Pauli complex. 
\end{proposition}

One can reformulate this proposition by saying that the map $\Pi\mapsto \delta_\Pi$ which to a flat $\Pi$ of the Pauli complex associates its unique root distribution is surjective onto the set of even root distributions.  In \cite{Pauli}, the flats of the Pauli complex were classified by means of  ring puzzles. The map $\Pi\mapsto \delta_\Pi$ can be reinterpreted as a map from ring puzzles to root distributions. It follows by Theorem 5.7 in \cite{Pauli} that the Pauli ring puzzles can be classified (i.e., listed explicitly); therefore, so can the even root distributions. 

Therefore,  one can deduce Theorem \ref{T - even root distribution classified}, and indeed, an explicit classification of the even root distributions, from the  classification of Pauli puzzles given in Theorem 5.7 in \cite{Pauli}. In the terminology  of \cite{Pauli}, the classification comprises:
\begin{enumerate}
\item a unique even root distribution associated with the Pauli $M$-puzzles; and,
\item an infinite family pairwise non-isomorphic  root distributions associated with the Pauli $T$-puzzles;
\end{enumerate}
every even root distribution is isomorphic to a root distribution in case (1) or (2). 
This statement is the more precise version of Theorem \ref{T - even root distribution classified}. As mentioned above, a more direct approach to this result, which classifies the even root distributions, is given in \ts\ref{S - classifying even roots}. 

\subsection{Even Moebius--Kantor complexes}
Several examples of even Moebius--Kantor complexes were given in \cite{rd}.  In fact,  it was shown there that there exist precisely four pairwise non isomorphic even Moebius--Kantor complexes having a single vertex. In the notation of \cite{rd}, these complexes are $V_0^1$, $V_0^2$, $\check V_0^2$, and $V_4^1$; we refer to \cite{rd} for their description (see also \cite[Prop.\ 3.2]{parity}).

\begin{theorem}\label{T - unique even} Suppose that $\Delta$ and $\Delta'$ are even Moebius--Kantor complex, and let $x\in \Delta$ and $x'\in \Delta'$ be two vertices. Let $n\geq 1$, and let $\p_n\colon B_n(x)\to B_n(x')$ be an isomorphism between the balls of radius $n$ around $x$ and $x'$ in $\Delta$ and $\Delta'$, respectively. Then there exists a  unique isomorphism $\p\colon \Delta\to \Delta'$ such that $\p_{|B_n(x)} = \p_n$. 
\end{theorem}

The following are immediate corollaries from Theorem \ref{T - unique even}.

\begin{corollary}
Every even simply connected Moebius--Kantor complex is isometrically isomorphic to the Pauli complex.
\end{corollary}

In fact, we will prove Theorem \ref{T - unique even} by cross referencing the additional structure found in the Pauli complex with that of an arbitrary even simply connected Moebius--Kantor complex.

The second corollary follows by choosing an appropriate isomorphism $\p_1\colon B_1(x)\to B_1(y)$ where $x$ and $y$ are two points in the Pauli complex:

\begin{corollary}
The symmetry group of the unique even simply connected Moebius--Kantor is flag transitive. 
\end{corollary} 

(We recall that a flag is a triple $(x,e,f)$ where $x$ is a vertex, $e$ an edge, and $f$ a face, such that $x\in e\in f$.) 

The following two corollaries were left open in \cite{rd} and \cite{Pauli}, respectively.

\begin{corollary}
The universal covers of $V_0^1$, $V_0^2$, $\check V_0^2$, and $V_4^1$, are pairwise isometrically isomorphic. 
\end{corollary}

We recall that $G_P$ denotes the fundamental group of the complex of groups defining the Pauli complex  (see \ts \ref{S - even root distr} and \ts\ref{Pauli}).

\begin{corollary}
The group $G_P$  contains a torsion free subgroup of finite index. 
\end{corollary}

\toc

\section{Abstract root distributions}\label{S - root prescription in flats}

\begin{proposition}\label{P - root distribution in a flat}
Every abstract root distribution can be realized in a Moebius--Kantor complex.  
\end{proposition}

\begin{proof}
Let $\Pi$ be a flat plane endowed with a root distribution. Fix a vertex $x\in \Pi$. Let $H_n$ denote the hexagon of center $x$ and radius $n$ in $\Pi$. The proof proceed by induction to construct a ball $B_n$ of radius $n$ in a Moebius--Kantor complex, containing $H_n$, and realizing the root distribution.  

For $H_1\subset B_1$, it suffices to embed a  6-cycle with marked roots in a Moebius--Kantor graph respecting the rank of every root. Suppose $H_n\subset B_n$ has been constructed. Let $\tilde B_n$ be obtained by adding a pair of faces for every edge in the boundary  $S_n$ of $B_n$, together with $H_{n+1}$, to $B_n$. For every vertex in  $S_n$, the link is either a tree (of diameter at most 5) or the union of such a tree and the 1-neighbourhood of a  6-cycle (adding at most two edges to the tree). In the first case, the tree can be completed arbitrarily into a Moebius--Kantor graph. In the second case, the roots are pre-positioned in the 6-cycle, and there exists a completion which respects the roots. 

Finally, the links can be completed to form $B_{n+1}$ from $\tilde B_n$, and the direct limit of the $B_n$'s constructed in a such a way provides one (among uncountably many non pairwise isomorphic) Moebius--Kantor complex realizing the root distribution.   
\end{proof}

Every root distribution defines uniquely the parity of its faces.

\begin{example}
A root distribution with parallel roots defines an even flat; however, the same flat can also be obtained with different root distributions.
\end{example}

We shall prove in the forthcoming section that there exists a parity distribution on a flat plane which is not realized by a root distribution. The notion of abstract root (and parity) distribution extends to simplicial subsets of $\IR^2$ (for example, strips or balls of radius $n$ centered at a vertex $x$) in the obvious way. Every parity (resp.\ root) distributions on a simplicial subset of $\IR^2$ is the restriction of a parity  (resp.\ root) distribution on $\IR^2$.

 Observe that every parity distribution on a bi-infinite strip of hight 1 can be realized in a flat in a Moebius--Kantor complex.  Indeed, it suffices to choose appropriately a root distribution on its boundary to realize the parity, extend the partial distribution in a flat, and apply Prop.\ \ref{P - root distribution in a flat}.
The same holds true for balls of radius 1 as we shall see in \ts\ref{S - parity one ball}.

\section{Abstract parity distributions}\label{S - parity vs roots}\label{S - parity root}

\begin{theorem}\label{L - abstract parity distribution in flat}
There exists an abstract parity distribution which is not the parity distribution of an abstract root distribution.
\end{theorem}

\begin{proof}
We shall start with an even face, say $x$, and  define a parity distribution stepwise   building outward from the face  $x$ in such a way that the root distribution is analytically determined, until a contradiction arises. 

Consider the following root distribution:
\begin{figure}[H]
\begin{tikzpicture}[scale=.7]

\begin{scope}[shift={(0,0)}]
\coordinate (O) at (0.0,0.0);
\coordinate (v1) at (1.0,0.0);
\coordinate (v2) at (0.5,0.87);
\coordinate (v3) at (-0.5,0.87);
\coordinate (v4) at (-1.0,-0.0);
\coordinate (v5) at (-0.5,-0.87);
\coordinate (v6) at (0.5,-0.87);
\draw[solid,thin,color=black] (v1) -- (v2) -- (v3)  -- (v4) -- (v5) -- (v6) -- (v1);
\draw[solid,thin,color=black] (v1) -- (v4);
\draw[solid,thin,color=black] (v2) -- (v5);
\draw[solid,thin,color=black] (v3) -- (v6);

\draw[solid,very thick,rotate=0,shift=+(O)] (-0.2,0) -- (.2,0);
\draw[solid,very thick,rotate=60,shift=+(v2)] (-0.2,0) -- (.2,0);
\draw[solid,very thick,rotate=0,shift=+(v3)] (-0.2,0) -- (.2,0);

\node[shift={(.35,-.25)}] (T) at (v3) {\tiny{$x$}};

\end{scope}
\end{tikzpicture}
\end{figure}
In the figure above, we represent a root distribution graphically as mentioned in the introduction (see \ts\ref{S - Intro root distribution}).

In the figure below, an extension of this root distrubution is shown. The triangles labelled with label ``1'' indicate the odd faces, every other face is even.   

\begin{figure}[H]
\begin{tikzpicture}[scale=.7]

\begin{scope}[shift={(0,0)}]
\coordinate (O) at (0.0,0.0);
\coordinate (v1) at (1.0,0.0);
\coordinate (v2) at (0.5,0.87);
\coordinate (v3) at (-0.5,0.87);
\coordinate (v4) at (-1.0,-0.0);
\coordinate (v5) at (-0.5,-0.87);
\coordinate (v6) at (0.5,-0.87);
\draw[solid,thin,color=black] (v1) -- (v2) -- (v3)  -- (v4) -- (v5) -- (v6) -- (v1);
\draw[solid,thin,color=black] (v1) -- (v4);
\draw[solid,thin,color=black] (v2) -- (v5);
\draw[solid,thin,color=black] (v3) -- (v6);

\draw[solid,very thick,rotate=0,shift=+(O)] (-0.2,0) -- (.2,0);
\draw[solid,very thick,rotate=60,shift=+(v1)] (-0.2,0) -- (.2,0);
\draw[solid,very thick,rotate=60,shift=+(v2)] (-0.2,0) -- (.2,0);
\draw[solid,very thick,rotate=0,shift=+(v3)] (-0.2,0) -- (.2,0);
\draw[solid,very thick,rotate=120,shift=+(v5)] (-0.2,0) -- (.2,0);
\draw[solid,very thick,rotate=0,shift=+(v6)] (-0.2,0) -- (.2,0);

\node[shift={(.35,-.25)}] (T) at (v3) {\tiny{$x$}};

\node[shift={(.35,.2)}] (T) at (v5) {\tiny{1}};
\end{scope}

\begin{scope}[shift={(0.5,-0.87)}]
\coordinate (O) at (0.0,0.0);
\coordinate (v1) at (1.0,0.0);
\coordinate (v2) at (0.5,0.87);
\coordinate (v3) at (-0.5,0.87);
\coordinate (v4) at (-1.0,-0.0);
\coordinate (v5) at (-0.5,-0.87);
\coordinate (v6) at (0.5,-0.87);
\draw[solid,thin,color=black] (v1) -- (v2) -- (v3)  -- (v4) -- (v5) -- (v6) -- (v1);
\draw[solid,thin,color=black] (v1) -- (v4);
\draw[solid,thin,color=black] (v2) -- (v5);
\draw[solid,thin,color=black] (v3) -- (v6);

\draw[solid,very thick,rotate=0,shift=+(O)] (-0.2,0) -- (.2,0);
\draw[solid,very thick,rotate=60,shift=+(v1)] (-0.2,0) -- (.2,0);
\draw[solid,very thick,rotate=0,shift=+(v5)] (-0.2,0) -- (.2,0);
\draw[solid,very thick,rotate=60,shift=+(v6)] (-0.2,0) -- (.2,0);

\node[shift={(.35,.2)}] (T) at (v5) {\tiny{1}};
\node[shift={(.35,-.25)}] (T) at (O) {\tiny{$y$}};

\end{scope}

\end{tikzpicture}
\end{figure}
It is obvious that the face with label $y$, which we define to be even, is odd according to the root distribution. This provides a contradiction.  

We may now extend the parity distribution by an order 3 symmetry as follows: 

\begin{figure}[H]
\begin{tikzpicture}[scale=.7]

\begin{scope}[shift={(0,0)}]
\coordinate (O) at (0.0,0.0);
\coordinate (v1) at (1.0,0.0);
\coordinate (v2) at (0.5,0.87);
\coordinate (v3) at (-0.5,0.87);
\coordinate (v4) at (-1.0,-0.0);
\coordinate (v5) at (-0.5,-0.87);
\coordinate (v6) at (0.5,-0.87);
\draw[solid,thin,color=black] (v1) -- (v2) -- (v3)  -- (v4) -- (v5) -- (v6) -- (v1);
\draw[solid,thin,color=black] (v1) -- (v4);
\draw[solid,thin,color=black] (v2) -- (v5);
\draw[solid,thin,color=black] (v3) -- (v6);

\node[shift={(.35,-.25)}] (T) at (v3) {\tiny{$x$}};

\node[shift={(.35,.2)}] (T) at (v5) {\tiny{1}};
\end{scope}

\begin{scope}[shift={(0.5,-0.87)}]
\coordinate (O) at (0.0,0.0);
\coordinate (v1) at (1.0,0.0);
\coordinate (v2) at (0.5,0.87);
\coordinate (v3) at (-0.5,0.87);
\coordinate (v4) at (-1.0,-0.0);
\coordinate (v5) at (-0.5,-0.87);
\coordinate (v6) at (0.5,-0.87);
\draw[solid,thin,color=black] (v1) -- (v2) -- (v3)  -- (v4) -- (v5) -- (v6) -- (v1);
\draw[solid,thin,color=black] (v1) -- (v4);
\draw[solid,thin,color=black] (v2) -- (v5);
\draw[solid,thin,color=black] (v3) -- (v6);

\node[shift={(.35,.2)}] (T) at (v5) {\tiny{1}};

\end{scope}

\begin{scope}[shift={(0.5,0.87)}]
\coordinate (O) at (0.0,0.0);
\coordinate (v1) at (1.0,0.0);
\coordinate (v2) at (0.5,0.87);
\coordinate (v3) at (-0.5,0.87);
\coordinate (v4) at (-1.0,-0.0);
\coordinate (v5) at (-0.5,-0.87);
\coordinate (v6) at (0.5,-0.87);
\draw[solid,thin,color=black] (v1) -- (v2) -- (v3)  -- (v4) -- (v5) -- (v6) -- (v1);
\draw[solid,thin,color=black] (v1) -- (v4);
\draw[solid,thin,color=black] (v2) -- (v5);
\draw[solid,thin,color=black] (v3) -- (v6);

\node[shift={(.35,.2)}] (T) at (O) {\tiny{1}};

\end{scope}
\begin{scope}[shift={(1,1.74)}]
\coordinate (O) at (0.0,0.0);
\coordinate (v1) at (1.0,0.0);
\coordinate (v2) at (0.5,0.87);
\coordinate (v3) at (-0.5,0.87);
\coordinate (v4) at (-1.0,-0.0);
\coordinate (v5) at (-0.5,-0.87);
\coordinate (v6) at (0.5,-0.87);
\draw[solid,thin,color=black] (v1) -- (v2) -- (v3)  -- (v4) -- (v5) -- (v6) -- (v1);
\draw[solid,thin,color=black] (v1) -- (v4);
\draw[solid,thin,color=black] (v2) -- (v5);
\draw[solid,thin,color=black] (v3) -- (v6);

\node[shift={(.35,.2)}] (T) at (O) {\tiny{1}};

\end{scope}

\begin{scope}[shift={(-0.5,0.87)}]
\coordinate (O) at (0.0,0.0);
\coordinate (v1) at (1.0,0.0);
\coordinate (v2) at (0.5,0.87);
\coordinate (v3) at (-0.5,0.87);
\coordinate (v4) at (-1.0,-0.0);
\coordinate (v5) at (-0.5,-0.87);
\coordinate (v6) at (0.5,-0.87);
\draw[solid,thin,color=black] (v1) -- (v2) -- (v3)  -- (v4) -- (v5) -- (v6) -- (v1);
\draw[solid,thin,color=black] (v1) -- (v4);
\draw[solid,thin,color=black] (v2) -- (v5);
\draw[solid,thin,color=black] (v3) -- (v6);

\node[shift={(.35,.2)}] (T) at (v4) {\tiny{1}};

\end{scope}

\begin{scope}[shift={(-1.5,0.87)}]
\coordinate (O) at (0.0,0.0);
\coordinate (v1) at (1.0,0.0);
\coordinate (v2) at (0.5,0.87);
\coordinate (v3) at (-0.5,0.87);
\coordinate (v4) at (-1.0,-0.0);
\coordinate (v5) at (-0.5,-0.87);
\coordinate (v6) at (0.5,-0.87);
\draw[solid,thin,color=black] (v1) -- (v2) -- (v3)  -- (v4) -- (v5) -- (v6) -- (v1);
\draw[solid,thin,color=black] (v1) -- (v4);
\draw[solid,thin,color=black] (v2) -- (v5);
\draw[solid,thin,color=black] (v3) -- (v6);

\node[shift={(.35,.2)}] (T) at (v4) {\tiny{1}};

\end{scope}

\end{tikzpicture}
\end{figure}
If follows that for every parity distribution on $\IR^2$ containing the above distribution, the twelve ($=3\times 4$) root distributions on $x$ in which the number of roots of rank 2 is 2, are disallowed. 

Finally, the remaining even root distribution on $x$:

\begin{figure}[H]
\begin{tikzpicture}[scale=.7]

\begin{scope}[shift={(0,0)}]
\coordinate (O) at (0.0,0.0);
\coordinate (v1) at (1.0,0.0);
\coordinate (v2) at (0.5,0.87);
\coordinate (v3) at (-0.5,0.87);
\coordinate (v4) at (-1.0,-0.0);
\coordinate (v5) at (-0.5,-0.87);
\coordinate (v6) at (0.5,-0.87);
\draw[solid,thin,color=black] (v1) -- (v2) -- (v3)  -- (v4) -- (v5) -- (v6) -- (v1);
\draw[solid,thin,color=black] (v1) -- (v4);
\draw[solid,thin,color=black] (v2) -- (v5);
\draw[solid,thin,color=black] (v3) -- (v6);

\draw[solid,very thick,rotate=0,shift=+(O)] (-0.2,0) -- (.2,0);
\draw[solid,very thick,rotate=120,shift=+(v2)] (-0.2,0) -- (.2,0);
\draw[solid,very thick,rotate=60,shift=+(v3)] (-0.2,0) -- (.2,0);

\node[shift={(.35,-.25)}] (T) at (v3) {\tiny{$x$}};

\end{scope}
\end{tikzpicture}
\end{figure}
\noindent is also disallowed, for the same reason (i.e., it extends as the initial root distribution). This concludes the proof of the proposition.
\end{proof}

\section{Parity prescription in the neighbourhood of a vertex} \label{S - parity one ball}

\begin{proposition}
Every parity distribution on a ball of radius 1 (hexagon) can be realized by an abstract root distribution. 
\end{proposition}

\begin{proof} Let $B$ be a ball of radius 1 around a vertex.  We distinguish three cases: either 
$B$ contains at most two odd faces, at most two even faces, or precisely three odd faces and three even faces.

The first two cases consists of five subcases each, which correspond to the number of ways to assign a parity to $B$ with at most two even (or odd) faces. In these cases, the root distributions can be chosen as follows:

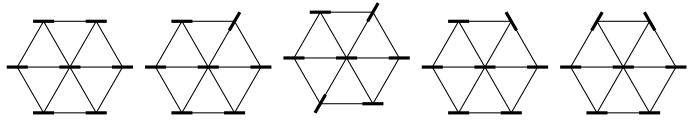
\begin{figure}[H]
\begin{tikzpicture}[scale=.7]
\coordinate (O) at (0.0,0.0);
\coordinate (v1) at (1.0,0.0);
\coordinate (v2) at (0.5,0.87);
\coordinate (v3) at (-0.5,0.87);
\coordinate (v4) at (-1.0,-0.0);
\coordinate (v5) at (-0.5,-0.87);
\coordinate (v6) at (0.5,-0.87);
\draw[solid,thin,color=black] (v1) -- (v2) -- (v3)  -- (v4) -- (v5) -- (v6) -- (v1);
\draw[solid,thin,color=black] (v1) -- (v4);
\draw[solid,thin,color=black] (v2) -- (v5);
\draw[solid,thin,color=black] (v3) -- (v6);

\draw[solid,very thick,rotate=0,shift=+(O)] (-0.2,0) -- (.2,0);
\draw[solid,very thick,rotate=0,shift=+(v1)] (-0.2,0) -- (.2,0);
\draw[solid,very thick,rotate=0,shift=+(v2)] (-0.2,0) -- (.2,0);
\draw[solid,very thick,rotate=0,shift=+(v3)] (-0.2,0) -- (.2,0);
\draw[solid,very thick,rotate=0,shift=+(v4)] (-0.2,0) -- (.2,0);
\draw[solid,very thick,rotate=0,shift=+(v5)] (-0.2,0) -- (.2,0);
\draw[solid,very thick,rotate=0,shift=+(v6)] (-0.2,0) -- (.2,0);
\end{tikzpicture}
\begin{tikzpicture}[scale=.7]
\coordinate (O) at (0.0,0.0);
\coordinate (v1) at (1.0,0.0);
\coordinate (v2) at (0.5,0.87);
\coordinate (v3) at (-0.5,0.87);
\coordinate (v4) at (-1.0,-0.0);
\coordinate (v5) at (-0.5,-0.87);
\coordinate (v6) at (0.5,-0.87);
\draw[solid,thin,color=black] (v1) -- (v2) -- (v3)  -- (v4) -- (v5) -- (v6) -- (v1);
\draw[solid,thin,color=black] (v1) -- (v4);
\draw[solid,thin,color=black] (v2) -- (v5);
\draw[solid,thin,color=black] (v3) -- (v6);

\draw[solid,very thick,rotate=0,shift=+(O)] (-0.2,0) -- (.2,0);
\draw[solid,very thick,rotate=0,shift=+(v1)] (-0.2,0) -- (.2,0);
\draw[solid,very thick,rotate=60,shift=+(v2)] (-0.2,0) -- (.2,0);
\draw[solid,very thick,rotate=0,shift=+(v3)] (-0.2,0) -- (.2,0);
\draw[solid,very thick,rotate=0,shift=+(v4)] (-0.2,0) -- (.2,0);
\draw[solid,very thick,rotate=0,shift=+(v5)] (-0.2,0) -- (.2,0);
\draw[solid,very thick,rotate=0,shift=+(v6)] (-0.2,0) -- (.2,0);
\end{tikzpicture}
\begin{tikzpicture}[scale=.7]
\coordinate (O) at (0.0,0.0);
\coordinate (v1) at (1.0,0.0);
\coordinate (v2) at (0.5,0.87);
\coordinate (v3) at (-0.5,0.87);
\coordinate (v4) at (-1.0,-0.0);
\coordinate (v5) at (-0.5,-0.87);
\coordinate (v6) at (0.5,-0.87);
\draw[solid,thin,color=black] (v1) -- (v2) -- (v3)  -- (v4) -- (v5) -- (v6) -- (v1);
\draw[solid,thin,color=black] (v1) -- (v4);
\draw[solid,thin,color=black] (v2) -- (v5);
\draw[solid,thin,color=black] (v3) -- (v6);

\draw[solid,very thick,rotate=0,shift=+(O)] (-0.2,0) -- (.2,0);
\draw[solid,very thick,rotate=0,shift=+(v1)] (-0.2,0) -- (.2,0);
\draw[solid,very thick,rotate=60,shift=+(v2)] (-0.2,0) -- (.2,0);
\draw[solid,very thick,rotate=0,shift=+(v3)] (-0.2,0) -- (.2,0);
\draw[solid,very thick,rotate=0,shift=+(v4)] (-0.2,0) -- (.2,0);
\draw[solid,very thick,rotate=60,shift=+(v5)] (-0.2,0) -- (.2,0);
\draw[solid,very thick,rotate=0,shift=+(v6)] (-0.2,0) -- (.2,0);
\end{tikzpicture}
\begin{tikzpicture}[scale=.7]
\coordinate (O) at (0.0,0.0);
\coordinate (v1) at (1.0,0.0);
\coordinate (v2) at (0.5,0.87);
\coordinate (v3) at (-0.5,0.87);
\coordinate (v4) at (-1.0,-0.0);
\coordinate (v5) at (-0.5,-0.87);
\coordinate (v6) at (0.5,-0.87);
\draw[solid,thin,color=black] (v1) -- (v2) -- (v3)  -- (v4) -- (v5) -- (v6) -- (v1);
\draw[solid,thin,color=black] (v1) -- (v4);
\draw[solid,thin,color=black] (v2) -- (v5);
\draw[solid,thin,color=black] (v3) -- (v6);

\draw[solid,very thick,rotate=0,shift=+(O)] (-0.2,0) -- (.2,0);
\draw[solid,very thick,rotate=0,shift=+(v1)] (-0.2,0) -- (.2,0);
\draw[solid,very thick,rotate=120,shift=+(v2)] (-0.2,0) -- (.2,0);
\draw[solid,very thick,rotate=0,shift=+(v3)] (-0.2,0) -- (.2,0);
\draw[solid,very thick,rotate=0,shift=+(v4)] (-0.2,0) -- (.2,0);
\draw[solid,very thick,rotate=0,shift=+(v5)] (-0.2,0) -- (.2,0);
\draw[solid,very thick,rotate=0,shift=+(v6)] (-0.2,0) -- (.2,0);
\end{tikzpicture}
\begin{tikzpicture}[scale=.7]
\coordinate (O) at (0.0,0.0);
\coordinate (v1) at (1.0,0.0);
\coordinate (v2) at (0.5,0.87);
\coordinate (v3) at (-0.5,0.87);
\coordinate (v4) at (-1.0,-0.0);
\coordinate (v5) at (-0.5,-0.87);
\coordinate (v6) at (0.5,-0.87);
\draw[solid,thin,color=black] (v1) -- (v2) -- (v3)  -- (v4) -- (v5) -- (v6) -- (v1);
\draw[solid,thin,color=black] (v1) -- (v4);
\draw[solid,thin,color=black] (v2) -- (v5);
\draw[solid,thin,color=black] (v3) -- (v6);

\draw[solid,very thick,rotate=0,shift=+(O)] (-0.2,0) -- (.2,0);
\draw[solid,very thick,rotate=0,shift=+(v1)] (-0.2,0) -- (.2,0);
\draw[solid,very thick,rotate=120,shift=+(v2)] (-0.2,0) -- (.2,0);
\draw[solid,very thick,rotate=60,shift=+(v3)] (-0.2,0) -- (.2,0);
\draw[solid,very thick,rotate=0,shift=+(v4)] (-0.2,0) -- (.2,0);
\draw[solid,very thick,rotate=0,shift=+(v5)] (-0.2,0) -- (.2,0);
\draw[solid,very thick,rotate=0,shift=+(v6)] (-0.2,0) -- (.2,0);
\end{tikzpicture}
\caption{Root distributions with at most 2 odd faces}
\end{figure}

It is easily seen that the above root distributions exhaust the parity distributions with at most two odd faces. Similarly:

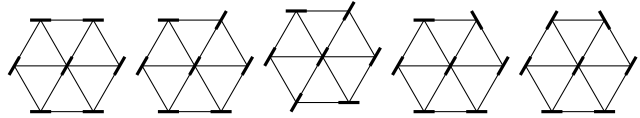
\begin{figure}[H]
\begin{tikzpicture}[scale=.7]
\coordinate (O) at (0.0,0.0);
\coordinate (v1) at (1.0,0.0);
\coordinate (v2) at (0.5,0.87);
\coordinate (v3) at (-0.5,0.87);
\coordinate (v4) at (-1.0,-0.0);
\coordinate (v5) at (-0.5,-0.87);
\coordinate (v6) at (0.5,-0.87);
\draw[solid,thin,color=black] (v1) -- (v2) -- (v3)  -- (v4) -- (v5) -- (v6) -- (v1);
\draw[solid,thin,color=black] (v1) -- (v4);
\draw[solid,thin,color=black] (v2) -- (v5);
\draw[solid,thin,color=black] (v3) -- (v6);

\draw[solid,very thick,rotate=60,shift=+(O)] (-0.2,0) -- (.2,0);
\draw[solid,very thick,rotate=60,shift=+(v1)] (-0.2,0) -- (.2,0);
\draw[solid,very thick,rotate=0,shift=+(v2)] (-0.2,0) -- (.2,0);
\draw[solid,very thick,rotate=0,shift=+(v3)] (-0.2,0) -- (.2,0);
\draw[solid,very thick,rotate=60,shift=+(v4)] (-0.2,0) -- (.2,0);
\draw[solid,very thick,rotate=0,shift=+(v5)] (-0.2,0) -- (.2,0);
\draw[solid,very thick,rotate=0,shift=+(v6)] (-0.2,0) -- (.2,0);
\end{tikzpicture}
\begin{tikzpicture}[scale=.7]
\coordinate (O) at (0.0,0.0);
\coordinate (v1) at (1.0,0.0);
\coordinate (v2) at (0.5,0.87);
\coordinate (v3) at (-0.5,0.87);
\coordinate (v4) at (-1.0,-0.0);
\coordinate (v5) at (-0.5,-0.87);
\coordinate (v6) at (0.5,-0.87);
\draw[solid,thin,color=black] (v1) -- (v2) -- (v3)  -- (v4) -- (v5) -- (v6) -- (v1);
\draw[solid,thin,color=black] (v1) -- (v4);
\draw[solid,thin,color=black] (v2) -- (v5);
\draw[solid,thin,color=black] (v3) -- (v6);

\draw[solid,very thick,rotate=60,shift=+(O)] (-0.2,0) -- (.2,0);
\draw[solid,very thick,rotate=60,shift=+(v1)] (-0.2,0) -- (.2,0);
\draw[solid,very thick,rotate=60,shift=+(v2)] (-0.2,0) -- (.2,0);
\draw[solid,very thick,rotate=0,shift=+(v3)] (-0.2,0) -- (.2,0);
\draw[solid,very thick,rotate=60,shift=+(v4)] (-0.2,0) -- (.2,0);
\draw[solid,very thick,rotate=0,shift=+(v5)] (-0.2,0) -- (.2,0);
\draw[solid,very thick,rotate=0,shift=+(v6)] (-0.2,0) -- (.2,0);
\end{tikzpicture}
\begin{tikzpicture}[scale=.7]
\coordinate (O) at (0.0,0.0);
\coordinate (v1) at (1.0,0.0);
\coordinate (v2) at (0.5,0.87);
\coordinate (v3) at (-0.5,0.87);
\coordinate (v4) at (-1.0,-0.0);
\coordinate (v5) at (-0.5,-0.87);
\coordinate (v6) at (0.5,-0.87);
\draw[solid,thin,color=black] (v1) -- (v2) -- (v3)  -- (v4) -- (v5) -- (v6) -- (v1);
\draw[solid,thin,color=black] (v1) -- (v4);
\draw[solid,thin,color=black] (v2) -- (v5);
\draw[solid,thin,color=black] (v3) -- (v6);

\draw[solid,very thick,rotate=60,shift=+(O)] (-0.2,0) -- (.2,0);
\draw[solid,very thick,rotate=60,shift=+(v1)] (-0.2,0) -- (.2,0);
\draw[solid,very thick,rotate=60,shift=+(v2)] (-0.2,0) -- (.2,0);
\draw[solid,very thick,rotate=0,shift=+(v3)] (-0.2,0) -- (.2,0);
\draw[solid,very thick,rotate=60,shift=+(v4)] (-0.2,0) -- (.2,0);
\draw[solid,very thick,rotate=60,shift=+(v5)] (-0.2,0) -- (.2,0);
\draw[solid,very thick,rotate=0,shift=+(v6)] (-0.2,0) -- (.2,0);
\end{tikzpicture}
\begin{tikzpicture}[scale=.7]
\coordinate (O) at (0.0,0.0);
\coordinate (v1) at (1.0,0.0);
\coordinate (v2) at (0.5,0.87);
\coordinate (v3) at (-0.5,0.87);
\coordinate (v4) at (-1.0,-0.0);
\coordinate (v5) at (-0.5,-0.87);
\coordinate (v6) at (0.5,-0.87);
\draw[solid,thin,color=black] (v1) -- (v2) -- (v3)  -- (v4) -- (v5) -- (v6) -- (v1);
\draw[solid,thin,color=black] (v1) -- (v4);
\draw[solid,thin,color=black] (v2) -- (v5);
\draw[solid,thin,color=black] (v3) -- (v6);

\draw[solid,very thick,rotate=60,shift=+(O)] (-0.2,0) -- (.2,0);
\draw[solid,very thick,rotate=60,shift=+(v1)] (-0.2,0) -- (.2,0);
\draw[solid,very thick,rotate=120,shift=+(v2)] (-0.2,0) -- (.2,0);
\draw[solid,very thick,rotate=0,shift=+(v3)] (-0.2,0) -- (.2,0);
\draw[solid,very thick,rotate=60,shift=+(v4)] (-0.2,0) -- (.2,0);
\draw[solid,very thick,rotate=0,shift=+(v5)] (-0.2,0) -- (.2,0);
\draw[solid,very thick,rotate=0,shift=+(v6)] (-0.2,0) -- (.2,0);
\end{tikzpicture}
\begin{tikzpicture}[scale=.7]
\coordinate (O) at (0.0,0.0);
\coordinate (v1) at (1.0,0.0);
\coordinate (v2) at (0.5,0.87);
\coordinate (v3) at (-0.5,0.87);
\coordinate (v4) at (-1.0,-0.0);
\coordinate (v5) at (-0.5,-0.87);
\coordinate (v6) at (0.5,-0.87);
\draw[solid,thin,color=black] (v1) -- (v2) -- (v3)  -- (v4) -- (v5) -- (v6) -- (v1);
\draw[solid,thin,color=black] (v1) -- (v4);
\draw[solid,thin,color=black] (v2) -- (v5);
\draw[solid,thin,color=black] (v3) -- (v6);

\draw[solid,very thick,rotate=60,shift=+(O)] (-0.2,0) -- (.2,0);
\draw[solid,very thick,rotate=60,shift=+(v1)] (-0.2,0) -- (.2,0);
\draw[solid,very thick,rotate=120,shift=+(v2)] (-0.2,0) -- (.2,0);
\draw[solid,very thick,rotate=60,shift=+(v3)] (-0.2,0) -- (.2,0);
\draw[solid,very thick,rotate=60,shift=+(v4)] (-0.2,0) -- (.2,0);
\draw[solid,very thick,rotate=0,shift=+(v5)] (-0.2,0) -- (.2,0);
\draw[solid,very thick,rotate=0,shift=+(v6)] (-0.2,0) -- (.2,0);
\end{tikzpicture}
\caption{Root distributions with at most 2 even faces}
\end{figure}

The last case, with 3 even faces and 3 odd faces, consists of 3 subcases, and for each of these cases, a root distribution exists: 
\begin{figure}[H]
\begin{tikzpicture}[scale=.7]
\coordinate (O) at (0.0,0.0);
\coordinate (v1) at (1.0,0.0);
\coordinate (v2) at (0.5,0.87);
\coordinate (v3) at (-0.5,0.87);
\coordinate (v4) at (-1.0,-0.0);
\coordinate (v5) at (-0.5,-0.87);
\coordinate (v6) at (0.5,-0.87);
\draw[solid,thin,color=black] (v1) -- (v2) -- (v3)  -- (v4) -- (v5) -- (v6) -- (v1);
\draw[solid,thin,color=black] (v1) -- (v4);
\draw[solid,thin,color=black] (v2) -- (v5);
\draw[solid,thin,color=black] (v3) -- (v6);

\draw[solid,very thick,rotate=0,shift=+(O)] (-0.2,0) -- (.2,0);
\draw[solid,very thick,rotate=0,shift=+(v1)] (-0.2,0) -- (.2,0);
\draw[solid,very thick,rotate=120,shift=+(v2)] (-0.2,0) -- (.2,0);
\draw[solid,very thick,rotate=60,shift=+(v3)] (-0.2,0) -- (.2,0);
\draw[solid,very thick,rotate=0,shift=+(v4)] (-0.2,0) -- (.2,0);
\draw[solid,very thick,rotate=0,shift=+(v5)] (-0.2,0) -- (.2,0);
\draw[solid,very thick,rotate=60,shift=+(v6)] (-0.2,0) -- (.2,0);
\end{tikzpicture}
\begin{tikzpicture}[scale=.7]
\coordinate (O) at (0.0,0.0);
\coordinate (v1) at (1.0,0.0);
\coordinate (v2) at (0.5,0.87);
\coordinate (v3) at (-0.5,0.87);
\coordinate (v4) at (-1.0,-0.0);
\coordinate (v5) at (-0.5,-0.87);
\coordinate (v6) at (0.5,-0.87);
\draw[solid,thin,color=black] (v1) -- (v2) -- (v3)  -- (v4) -- (v5) -- (v6) -- (v1);
\draw[solid,thin,color=black] (v1) -- (v4);
\draw[solid,thin,color=black] (v2) -- (v5);
\draw[solid,thin,color=black] (v3) -- (v6);

\draw[solid,very thick,rotate=0,shift=+(O)] (-0.2,0) -- (.2,0);
\draw[solid,very thick,rotate=0,shift=+(v1)] (-0.2,0) -- (.2,0);
\draw[solid,very thick,rotate=120,shift=+(v2)] (-0.2,0) -- (.2,0);
\draw[solid,very thick,rotate=0,shift=+(v3)] (-0.2,0) -- (.2,0);
\draw[solid,very thick,rotate=0,shift=+(v4)] (-0.2,0) -- (.2,0);
\draw[solid,very thick,rotate=60,shift=+(v5)] (-0.2,0) -- (.2,0);
\draw[solid,very thick,rotate=0,shift=+(v6)] (-0.2,0) -- (.2,0);
\end{tikzpicture}
\begin{tikzpicture}[scale=.7]
\coordinate (O) at (0.0,0.0);
\coordinate (v1) at (1.0,0.0);
\coordinate (v2) at (0.5,0.87);
\coordinate (v3) at (-0.5,0.87);
\coordinate (v4) at (-1.0,-0.0);
\coordinate (v5) at (-0.5,-0.87);
\coordinate (v6) at (0.5,-0.87);
\draw[solid,thin,color=black] (v1) -- (v2) -- (v3)  -- (v4) -- (v5) -- (v6) -- (v1);
\draw[solid,thin,color=black] (v1) -- (v4);
\draw[solid,thin,color=black] (v2) -- (v5);
\draw[solid,thin,color=black] (v3) -- (v6);

\draw[solid,very thick,rotate=0,shift=+(O)] (-0.2,0) -- (.2,0);
\draw[solid,very thick,rotate=0,shift=+(v1)] (-0.2,0) -- (.2,0);
\draw[solid,very thick,rotate=120,shift=+(v2)] (-0.2,0) -- (.2,0);
\draw[solid,very thick,rotate=120,shift=+(v3)] (-0.2,0) -- (.2,0);
\draw[solid,very thick,rotate=0,shift=+(v4)] (-0.2,0) -- (.2,0);
\draw[solid,very thick,rotate=0,shift=+(v5)] (-0.2,0) -- (.2,0);
\draw[solid,very thick,rotate=0,shift=+(v6)] (-0.2,0) -- (.2,0);
\end{tikzpicture}
\caption{Root distributions with exactly three odd faces}
\end{figure}
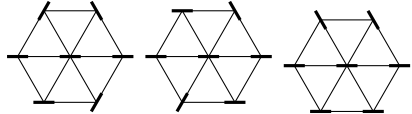
\noindent This concludes the proof of the proposition. \end{proof}

\section{The Pauli complex $\Delta_P$}\label{Pauli}

The Pauli complex is a simply connected Moebius--Kantor complex $\Delta_P$ obtained by a standard construction \cite[Chap.\ 3]{BH} using a triangle of groups  associated with the group $P$ generated by the three Pauli matrices in $\SU(2)$:
\[
X := \begin{pmatrix}0&1\\1&0\end{pmatrix},\ \   Y := \begin{pmatrix}0&-i\\i&0\end{pmatrix},\ \   \text{ and } \ \ Z := \begin{pmatrix}1&0\\0&-1\end{pmatrix}.
\]
The direct limit $G_P$ of this triangle of groups acts on $X_P$ in a flag transitive way. The complex $\Delta_P$ is obtained by developing around a fundamental face $f$ using the group $P$. 

The complex $\Delta_P$ has an important property, which was  used in \cite{Pauli} to classify the flat planes that it contains: every face in  $\Delta_P$ can be assigned with a Pauli matrix $X$, $Y$ or $Z$ in an equivariant way \cite[\ts 3]{Pauli}.  Under this assignment, the link edges are naturally labeled, and the link is a Moebius--Kantor graph  isomorphic to the Cayley graph of $P$ with respect to $\{X,Y,Z\}$. The 6-cycles in the link (which correspond to rings in the ring puzzles) are given by the following relations between the Pauli matrices
\begin{enumerate}
\item
[] $XYXZYZ=\Id$
\item
[] $YZYXZX=\Id$
\item
[] $ZXZYXY=\Id$
\end{enumerate}
in $P$. A root in $\Delta_P$ corresponds to a path of length 3 in the link  with an edge labelling in $\{X,Y,Z\}$.   

\begin{lemma}\label{L - rank 2 roots}
A root is of rank 2 if and only if it contains the three labels $X$, $Y$, $Z$.
\end{lemma}

\begin{proof}
Since the Moebius--Kantor graph is transitive on path of length 2, it is enough to prove the lemma for the roots having labels $XYZ$ and $XYX$. It is clear that a root with labels $XYZ$ can be extended in two ways using the 6-cycles $XYZXZY$ and $XYZYXZ$, and therefore is of rank 2. A root with labels $XYX$ can be extended in a single way using a 6-cycle, namely, $XYXZYZ$. The word $XYXY$ is not a subword of a relation of length 6 for the group $P$.
\end{proof}

This implies:

\begin{proposition}
The Pauli complex $\Delta_P$ is an even Moebius--Kantor complex.
\end{proposition}

\begin{proof}
Since $\Delta_P$ is flag transitive, it is enough to prove that the fundamental face is even. We assume that the face $f$ is labelled with the matrix $Y$, and extend $f$ into an equilateral triangle $\tilde f$ using three faces with the same labels, say $X$. Then the three roots associated with $f$ have labels $XYX$. By the previous lemma, these roots are not of rank 2. Therefore, $\Delta_P$ is even.     
\end{proof}

As shown in the introduction, in order to deduce Theorem \ref{T - even root distribution classified} from \cite{Pauli}, it is sufficient to prove the following result.

\begin{proposition}\label{P - even is realized in Pauli}
Every even root distribution can be realized in the Pauli complex. 
\end{proposition}

One argument would be to modify the proof of Prop.\ \ref{P - root distribution in a flat}, to ensure that the Moebius--Kantor complex is even, and apply Theorem \ref{T - unique even}.

We shall proceed as follows.

\begin{proposition}\label{P - even unique Pauli puzzle}
 Let $\Pi$ be a flat plane endowed with an even root distribution, and consider a labelling by $X$ and $Y$, of two consecutive faces in $\Pi$. There exists a unique labelling of the faces of $\Pi$ with labels in $\{X,Y,Z\}$ such that  the 6-cycles in the links correspond to the relations $XYXZYZ=\Id$, $YZYXZX=\Id$, and $ZXZYXY=\Id$.
In other words, there exists a unique Pauli puzzle associated with $\Pi$ and the labelling of two consecutive faces. 
\end{proposition}

Prop.\ \ref{P - even is realized in Pauli} follows from Prop.\ \ref{P - even unique Pauli puzzle} and the fact that every Pauli puzzle is realized in the Pauli complex (Theorem 3.3 in \cite{Pauli}).

\begin{proof}
Uniqueness is clear. We prove existence by induction on the radius $n$ of the ball $B_n$ centered at one of the two common vertices of the labelled consecutive triangles. 

Suppose that there exists a labelling of $B_n$. For every vertex $x$ in the boundary of $B_n$, at least two faces of the ball $B_1(x)$ of radius 1 at $x$ are labelled, and therefore there exists a unique labelling of $B_1(x)$ which is consistent with  the parity at $x$.

Let $[x,y]$ be an edge in the boundary of $B_n$, and let $t$ and $t'$ denotes the two triangles adjacent to $[x,y]$, where $t$ is the triangle included in $B_n$.

 We have to show that the labelling of $t'$ from $B_1(x)$  coincides with the labelling from $B_1(y)$. However, since two of the three faces adjacent to $t$ are included in $B_n$ and therefore labelled, only one of the two possible labels of $t'$ ensure, by Lemma \ref{L - rank 2 roots}, that $t$ is an even face (compare Lemma \ref{L - label well defined}).   
\end{proof}

\section{Every even complex is isomorphic to $\Delta_P$}\label{S - uniqueness}

Let $\Delta$ be a simply connected even complex, $x_0\in \Delta$ be a vertex, $B_n$ be the ball of center $x_0$ and radius $n$ in $\Delta$, and $S_n:=\del B_n$ be the sphere of center $x_0$ and 
radius $n$. We fix an isometric embedding
\[
\p_1\colon B_1\to \Delta_P.
\]
Using $\p_1$, we may label every face in $B_1$ with one of the Pauli matrices $X,Y,Z$.

\begin{theorem}
There exists a unique isometric embedding $\p\colon \Delta\to \Delta_P$ such that ${\p}_{|B_1}=\p_1$.
\end{theorem}

\begin{proof} We prove by induction that there exists a unique isometric embedding $\p_n\colon B_n\to \Delta_P$ such that ${\p_n}_{|B_1}=\p_1$. The theorem follows immediately.

Suppose that $\p_n\colon B_n\to \Delta_P$ is an isometric isomorphism onto its image. Using $\p_n$, we may label every face in $B_n$ with one of the Pauli matrices $X,Y,Z$. Let us construct $\p_{n+1}$. This construction will also establish Theorem \ref{T - unique even}.

For every vertex $x\in S_n$, there exists a unique isometric embedding $\p_x\colon L_x\to \Delta$ whose restriction to $B_n$ coincides with $\p_n$.

For every edge $[x,y]$ in $S_n$, let $h_{x,y}^k$, $k=0,1$, denote the height, perpendicular to $[x,y]$, in the two faces containing $[x,y]$ which are not contained in $B_n$. We write $H_{x,y}^k$ and $H_{y,x}^k$ for the two half faces, separated by $h_{x,y}^k$, respectively containing $x$ and $y$. Thus, $ H_{x,y}^k\cup H_{y,x}^k$, $k=0,1$, are the two faces on $[x,y]$ not contained in $B_n$.

For every $x\in S_n$ we may label, using $\p_x$,  every half face $H_{x,y}^k$, $k=0,1$, with one of the Pauli matrices $X,Y,Z$. 

We claim (Lemma \ref{L - label well defined}): 
\begin{quote} for every $[x,y]$ and $k=0,1$, the labels of the half faces $H_{x,y}^k$ and $H_{y,x}^k$ coincide.
\end{quote} 

This shows that there exists a well-defined isometric embedding $\p_{n+1}\colon B_{n+1}\to \Delta$ extending $\p_n$, which coincides with $\p_x$ on $L_x$ for every $x\in S_n$. It takes $H_{x,y}^k\cup H_{y,x}^k$ to the unique face in $\Delta_P$ which contains $\p_n([x,y])$, and whose label is the common label of $ H_{x,y}^k$ and $H_{y,x}^k$.  The uniqueness of such an embedding is clear. 
\end{proof}

We now prove the claim.

\begin{lemma}\label{L - label well defined} For every $[x,y]$ and $k=0,1$, the labels of the half faces $H_{x,y}^k$ and $H_{y,x}^k$ coincide.
\end{lemma}

\begin{proof}
Consider the face $[x,y,z]$ in $B_n$ containing $[x,y]$, and two faces $F_x$ and $F_y$ in $B_n$, adjacent to $[z,x]$ and $[z,y]$ respectively, and having identical labels, say $X$. Let $k,l\in\{0,1\}$ denote the two indices such that $H_{x,y}^k$ and $H_{y,x}^l$ have label $X$. We must show that $k=l$.

Since the labels of the faces $F_x$, $F_y$, and half faces $H_{x,y}^k$ and $H_{y,x}^l$, coincide, the rank of the roots corresponding three roots at  $x$, $y$ and $z$ is $\frac 3 2$ (Lemma \ref{L - rank 2 roots}).
 
Since $F$ is an even face, the number of roots of rank $\frac 3 2$ in a large triangle containing $F$ is odd. Since the root at $z$ is of rank $\frac 3 2$, the roots at $x$ and $y$ must have the same rank. Therefore, $k=l$. For the corresponding face in $\Delta$, the large triangle adjoins two roots of rank $\frac 3 2$ at $x$ and $y$ for a total of 3 such roots.   
\end{proof}

\section{Classifying the even root distributions}\label{S - classifying even roots}

In this section, we provide a direct approach to the classification of even root distributions in a flat, which rely only on the root distribution rather than the Pauli complex. 

We will in fact start with an arbitrary even Moebius--Kantor complex, not a priori known to be isomorphic to $\Delta_P$, and show how its flats can be classified by studying the root distributions. The same proof applies to classify the abstract even root distributions.

In \ts 2 of \cite{autf2puzzles}, a notion of $w$-block in an  $\Aut(F_2)$-puzzle was introduced. It is a geometric configuration of shapes having the property of being forward analytic in $\Aut(F_2)$ puzzles. We show here that a configuration leading to a similar behaviour can be found in every even simply connected Moebius--Kantor complex.

Let  $\Delta$ be an even simply connected Moebius--Kantor complex.

\begin{lemma}\label{L - t}  Consider a  trapezoid $t$ of height 1 with a base of length 3, such that the roots on the base have rank 2, and the root on the  top, rank 3/2. If $t$ belongs to a flat $\Pi$, then  $t$ admits a unique extension into a sector of $\Pi$ of base $t$ and rank 2 boundary.
\end{lemma}

\begin{proof} The rank 3/2 root in  $t$ extends in a unique way. Since $\Delta$ is even, the slanted sides of length 2 are of rank 3/2. This follows from the fact that the inner vertices are both of rank 2. Thus, $t$ extends uniquely into a 2-strip $S_2$ of length 3 in $\Pi$. Since $\Delta$ is even, the two upper roots in $\Pi$ are of rank 3/2, and therefore extend uniquely into $\Pi$. Again, the slanted sides of length 2 are of rank 3/2, and $S_2$ extends uniquely into $\Pi$ into a partial cone $S_3$ of height 3. By induction, the upper roots in $S_n$ are always of rank $3/2$, which provides a uniquely defined sector $S:=\bigcup S_i$ of base $t$ in $\Pi$.
\end{proof}

\begin{figure}[H]\includegraphics[width=7cm]{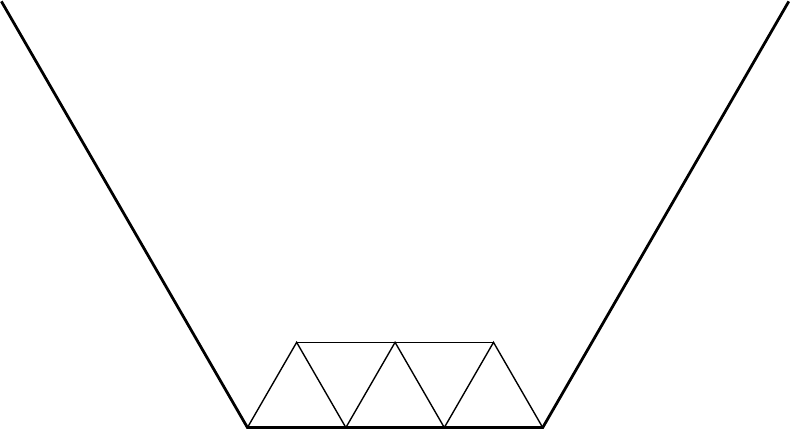}
\caption{The figure shows the sector $S$ and  its rank 2 boundary}
\end{figure}

\begin{lemma}\label{L - S}
The sector $S$ (in the notation of Lemma \ref{L - t}) belongs to a unique flat.
\end{lemma}

\begin{proof}
Let $S\subset \Pi$ be a copy of $S$ in a flat, and let $\overline t$ denote the reflection of $t$ with respect to its rank 2 boundary.  If the lower boundary of $\overline t$ is of rank 3/2, we have an opposite copy of $t$ in $\Pi$, and Lemma \ref{L - t} shows that $\overline t$ extends to a copy $\overline S$ of $S$ in $\Pi$. By convexity, there exists at most one flat of $\Delta$ containing $S\cup \overline S$. It is easy to check that such a flat indeed exists. (It is represented below.)

Let us show that the lower boundary of $\overline t$ must be of rank 3/2. Otherwise, it is of rank 2. Since the central triangle in $t$ is even, the 2-triangle $T$ (which is formed of 4 equilateral triangles) of base the lower base of $\overline t$ admits two sides of rank 3/2 or two sides of rank 2. Since the center triangle in $\overline t$ is even, this proves that  the lower side of $T$ is of rank 3/2. 
\end{proof}

\begin{figure}[H]\includegraphics[width=7cm]{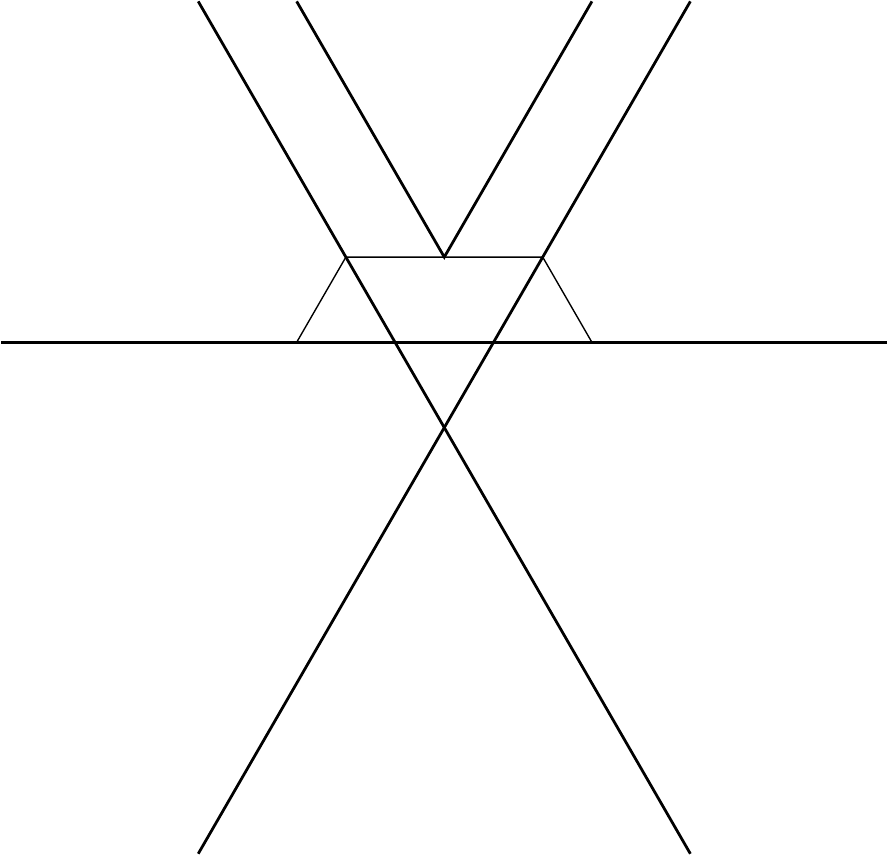}
\caption{The $t$-flat $\Pi_t$; three intersecting rank 2 (in dark) geodesics are drawn, and one of the trapezoid $t$ is shown; the symmetries $S_3$ of the center triangle extend to the $t$-flat.}
\end{figure}

\begin{remark}
The flat $\Pi_t$ corresponds to one of the Pauli $M$-puzzles in \cite{Pauli}, where $M\in \{X,Y,Z\}$.
\end{remark}

By strip of rank 2, we mean a simplicial flat strip whose boundary roots are of rank 2.  

\begin{lemma}
If $\Delta$ is even, the minimal strips of rank 2 are of height 1.
\end{lemma}

\begin{proof}
Suppose that there exists a minimal strip $\Sigma$ with rank 2 boundaries. If it is not of height 1, then there exists a root of rank 3/2 at distance 1 from the boundary. Since the boundary is of rank 2, $\Sigma$ contains a trapezoid $t$.  By Lemma \ref{L - t} and Lemma \ref{L - S}, $\Sigma$ must be included in the $t$-flat. However, the latter  does not contain such a strip. 
\end{proof}

By geodesic of rank 2 in a flat, we mean a simplicial straight line containing exclusively roots of rank 2.

\begin{lemma}
If a flat $\Pi$ contains a geodesic $l$ of rank 2, then either $\Pi$ coincides with the $t$-flat $\Pi_t$, or it is a union of minimal strips of rank 2 parallel to $l$. 
\end{lemma}

\begin{proof}
Let $l$ be a rank 2 geodesic. If $\Pi$ is not a union of minimal strips of rank 2 parallel to $l$, then there exists a parallel geodesic $l'\parallel l$ of rank 2, and  a parallel geodesic at distance 1 from $l'$ which contains a root of rank 3/2. Thus,  $\Pi$ contains the trapezoid $t$ which implies $\Pi=\Pi_t$.  
\end{proof}

 We shall now consider another `glider' (i.e., a configuration which is forward analytic in flats) which behaves similarly to the $w$-block in $\Aut(F_2)$ puzzles. Furthermore, if $\Pi\neq \Pi_t$ then it will also be backward analytic.

\begin{lemma}\label{L - t'} Let $\Delta$ be an even simply connected Moebius--Kantor complex. Consider a  trapezoid $t'$ of height 1 with a base of length 3, such that precisely one of the two roots on the base have rank 2, and such that the top root has rank 3/2. Every flat containing $t'$ contains a geodesic of rank 2.  
\end{lemma}

\begin{proof}
Suppose that $t'$ belongs to a flat $\Pi$. We first show that  $t'$ admits a unique upward extension into a half-strip $S^{\up}$ transverse to $t'$ of height 3 with rank 2 boundaries. By symmetry, we may assume that the bottom rank 2 root is on the right.  Consider the unique upward extension of $t'$ (whose boundary is of rank 3/2). Since the left top face of the trapezoid is even, the righthand side of the extension is of rank 3/2, and therefore admits a unique extension into $\Pi$. This implies that the top root of the resulting extension is of rank 3/2. Therefore, a copy of $t'$ sits on top of it; by induction, we find a half strip $S^{\up}$ piling up infinitely many copies of $t'$.

Next we prove that if $\Pi\neq \Pi_t$ then $S^{\up}$ extends uniquely into a strip $S$ of height 3 with rank 2 boundary. Consider a trapezoid $t''$ extending $S^{\up}$ downwards. Since the faces of $t'$ are even, the right bottom root of $t''$ is of rank 2. If the left bottom root of $t''$ is even, then $t''$ is isomorphic to $t$, so $\Pi=\Pi_t$ by Lemmas \ref{L - t} and \ref{L - S}.  Therefore $t''$ is isomorphic to $t'$; by induction,  we find a strip $S$ containing $S^{\up}$ and piling up infinitely many copies of $t'$.  

To prove the lemma, we may assume that $\Pi\neq \Pi_t$;  note then that the strip $S\subset \Pi$ is a union of strips with rank 2 boundaries.    
\end{proof}

\begin{theorem}
A flat plane in an even simply connected Moebius--Kantor complex is either a flat with the root distribution of $\Pi_t$, or a union of strips of height 1 and rank 2.    
\end{theorem}

\begin{proof}
Let $\Delta$ be an even Moebius--Kantor complex and $\Pi$ be a flat in $\Delta$. The previous results show that the conclusions hold if $\Pi$ contains $t$ (see Lemma \ref{L - t}) or $t'$ (see Lemma \ref{L - t'}). Otherwise, $\Pi$ contains neither $t$ nor $t'$. This implies the following statement: 

\begin{enumerate}
\item[] If $\tau$ is a trapezoid of base of length 3 and top root of rank 3/2, then the bottom roots are of rank 3/2.
\end{enumerate} 

By connectedness, this statement in turn implies that the roots of rank 3/2 are aligned in $\Pi$. There exists a unique configuration with this property. Since the roots of rank 3/2 are aligned in $\Pi$ and $\Delta$ is a Moebius--Kantor complex, the two transverse directions must be rank 2. Thus, $\Pi$ is a union of  strips of height 1 and rank 2.
\end{proof}

\end{document}